\documentclass[12pt, a4paper]{article}
\usepackage{amsfonts,amssymb,amsmath,amscd,latexsym,makeidx,theorem,graphics}
\usepackage{hyperref}
\usepackage{color}

\usepackage[pdftex]{graphicx}

\newtheorem{theorem}{Theorem}[section]
\newtheorem{lemma}[theorem]{Lemma}

\newenvironment{proof}{\noindent\emph{Proof.}}{\hfill$\square$\medskip}

\newcommand{\D}{\Delta}

\newcommand{\R}{\mathbb{R}}
\newcommand{\ve}{\varepsilon}

\title{Higher Order Conformally Invariant Equations  in ${\mathbb R}^3$  with Prescribed Volume} 
\author{Ali Hyder\thanks{The author is supported by the Swiss National Science Foundation, Grant No. P2BSP2-172064}\\ {\small Department of Mathematics, University of British Columbia, Vancouver BC V6T1Z2, Canada}\\ {\small \texttt{ali.hyder@math.ubc.ca} }\and Juncheng Wei\thanks{ The research is partially supported by NSERC} \\  {\small  Department of Mathematics, University of British Columbia, Vancouver BC V6T1Z2, Canada}\\  {\small \texttt{jcwei@math.ubc.ca} } }

\date{}

\begin{document}
\maketitle

\abstract
In this paper we study  the following conformally invariant  poly-harmonic equation  $$\D^mu=-u^\frac{3+2m}{3-2m}\quad\text{in }\mathbb{R}^3,\quad u>0,$$ with $m=2,3$. We prove the existence of  positive smooth  radial solutions with prescribed volume  $\int_{\mathbb{R}^3} u^\frac{6}{3-2m}dx$. We show that the set of all  possible values of the  volume is a bounded interval $(0,\Lambda^*]$ for $m=2$, and it is $(0,\infty)$ for $m=3$. This is in sharp contrast to $m=1$ case  in which the volume $\int_{\mathbb{R}^3} u^\frac{6}{3-2m}dx$ is a fixed value.

\section{Introduction to the problem}  We consider the negative exponent problem 

\begin{equation}\label{eq-main} \D^mu=-u^\frac{3+2m}{3-2m}\quad\text{in }\R^3,\quad u>0,\end{equation} 
where $m$ is either $2$ or $3$. Geometrically, if $u$ is a smooth solution to \eqref{eq-main} then the conformal metric $g_u:=u^\frac{4}{3-2m}|dx|^2$ ($|dx|^2$ is the Euclidean metric on $\R^3$) has constant $Q$-curvature on $\R^3$, see \cite{Branson, CGS,Choi-Xu, Ngo-6,XY, Yang-Zhu}.  Moreover, the volume of the metric $g_u$ is $$\int_{\R^3}dV_{g_u}=\int_{\R^3}\sqrt{|g_u|}dx=\int_{\R^3}u^\frac{6}{3-2m}dx,$$  which is invariant under the scaling $u_\lambda(x):=\lambda^\frac{3-2m}{2}u(\lambda x)$ with $\lambda>0$.

 Equation    \eqref{eq-main} belongs to the class of conformally invariant equations. When $m=1$ this is called Yamabe equation; while for $m=2$ it is $Q$-curvature equation. In recent years Problem \eqref{eq-main} has been extensively studied in \cite{Choi-Xu, Ngo2,Guerra, HW,Lai,MR,Xu} for $m=2$, in \cite{Ngo-6} for $m=3$ and in \cite{F-X,Li,Ngo} for higher order case (but to an integral equation). We recall that radial solutions to \eqref{eq-main} with $m=2$ has  either  exactly liner growth or exactly quadratic growth at infinity, that is,    $$\lim_{r\to\infty}\frac{u(r)}{r}\in (0,\infty)\quad\text{or }\lim_{r\to\infty}\frac{u(r)}{r^2}\in(0,\infty).$$    The solution with exactly linear growth is unique (up to a scaling) and is given by    \begin{align}\label{U4}  U_0(r)=\sqrt{\sqrt{1/15}+r^2}.\end{align}    However, there are infinitely many (radial or nonradial) solutions with quadratic growth, see \cite{Ngo2,Guerra,HW}.
For $m=3$, radial solutions grow either cubically or quatrically at infinity, that is,    $$\lim_{r\to\infty}\frac{u(r)}{r^3}\in (0,\infty)\quad\text{or }\lim_{r\to\infty}\frac{u(r)}{r^4}\in(0,\infty).$$    In this case also we have an explicit solution which grows cubically at infinity, namely     $$U_1(r)=\left( 315^{-\frac13}+r^2\right)^\frac32.$$    It is worth pointing out that both solutions  $U_0$ and $U_1$ can be obtained by pulling back the round metric of $S^3$ via stereographic projection, and they  satisfy an integral equation of the form    $$U(x)=c_m\int_{\R^3}|x-y|^{p}U^\frac{3+2m}{3-2m}(y)dy,$$    where $p=1$ for $m=2$ and $p=3$ for $m=3$.  Nevertheless, $U_1$ is not unique (up to scaling) among the  radial solutions having exactly cubic growth at infinity.

We now state our main results concerning the existence of radial solutions to \eqref{eq-main} with prescribed volume.  For $m=2$ we prove:
\begin{theorem}\label{thm4th}
There exists  a radial solution to    \begin{align}\label{eq-4th} \D^2 u=-\frac{1}{u^7}\quad\text{in }\R^3,\quad u>0, \quad \Lambda_u:=\int_{\R^3}\frac{dx}{u^6(x)}\end{align}     if and  only if $\Lambda_u\in (0,\Lambda^*]$, where $\Lambda^*$ is the  volume of the metric $g_{U_0}$, that is,    \begin{align}\label{vol4} \Lambda^*:=\int_{\R^3}\frac{dx}{U_0^6(x)}=\int_{\R^3}\frac{dx}{(\sqrt{1/15}+|x|^2)^3}.\end{align}     Moreover, if  $\Lambda_u=\Lambda^*$  for some radial solution $u$ to \eqref{eq-4th} then up to a scaling we have $u=U_0$.  
\end{theorem}%

For $m=3$ we prove the existence of radial solution for every prescribed volume.
\begin{theorem}\label{thm-large} For every $\Lambda>0$ there exists a positive radial solution to
    \begin{align}\label{eq-6th}\D^3u=-\frac{1}{u^3} \quad\text{in } \R^3\end{align} such that \begin{align}\label{vol}\int_{\R^3}\frac{dx}{u^2(x)}=\Lambda.\end{align}   
\end{theorem}

A similar phenomena has already been exhibited in a higher order Liouville equation, namely    \begin{align}\label{liou} (-\D )^\frac n2 u=(n-1)!e^{nu}\quad\text{in }\R^n,\quad V:=\int_{\R^n}e^{nu}dx<\infty .\end{align}    (Here $V$ is the volume of the conformal metric $g_u=e^{2u}|dx|^2$). More precisely, if $u$ is a solution to \eqref{liou} with $n=4$ then necessarily $V\in (0,V^*]$, and $V=V^*$ if and only if $u$ is a spherical solution, that is, for some $\lambda>0$ and $x_0\in\R^n$ we have     $$u(x)=u_{\lambda,x_0}(x):=\log\left(\frac{2\lambda}{1+\lambda^2|x-x_0|^2}\right).$$     However, if $n\geq 5$ then  for every $V\in (0,\infty)$ there exists a radial solution to \eqref{liou}.  See \cite{CC,HY,H,Lin,M,WY} and the references therein.


Finally, we remark that the upper bound of $V$ in \eqref{liou} with $n=4$ comes from a Pohozaev type identity, and it holds for every solutions to \eqref{liou} (radial and non-radial). However, from a similar Pohozaev type identity one does not get the same conclusion on the volume of the metric  $g_u:=u^\frac{4}{3-2m}$, compare \cite[Lemma 2.3]{HW}.

\medskip

\noindent\textbf{Notations} For a radially symmetric function $u$ we will write $u(|x|)$ to denote the same function $u(x)$.

\section{Proof of the theorems}
We shall use the following comparison lemma of two radial solutions to $\D^n u=f(u)$, whose proof follows from  the ODE local uniqueness theorem, and a repeated use of  the identity \eqref{formula-1}. See also Lemma 3.2 in \cite{MR} and Proposition A.2 in \cite{Farina}. 
\begin{lemma}\label{comp} Let $f$ be a locally Lipschitz continuous and  monotone increasing function on $(0,\infty)$.  Let $u_1,u_2\in C^{2k}([0,R))$ be two positive solutions of    \begin{align*}\left\{ \begin{array}{ll} \D^ku=f(u)&\quad\text{on }(0,R)\\  \D^ju_1(0)\geq \D^j u_2(0)&\quad\text{for every }j\in J\\ (\D^ju_1)'(0)= (\D^j u_2)'(0)=0&\quad\text{for every }j\in J,\end{array}\right.\end{align*}    where $J:=\{0,1,\dots,k-1\} $.  Then $\D ^ju_1\geq \D^ju_2$ and $(\D^ju_1)'\geq (\D^j u_2)'$ on $(0,R)$ for every $j\in J $. Moreover, if $\D^ju_1(0)>\D^ju_2(0)$ for some $j\in J$ then   $\D ^ju_1> \D^ju_2$ and $(\D^ju_1)'> (\D^j u_2)'$ on $(0,R)$ for every $j\in J $.    \end{lemma}

With  the help of above comparison lemma and the fact that $\D U_0(\infty)=0$ we prove Theorem \ref{thm4th}.

\medskip

\noindent{\bf\emph{Proof of Theorem \ref{thm4th}}} For $\rho\in (-U_0(0),\infty)$ we consider the   initial value problem     \begin{align}\label{urho}\left\{ \begin{array}{ll} \D^2u_\rho=-\frac{1}{u_\rho^7} \\  u_\rho(0)=U_0(0)+\rho\\  \D u_\rho(0)=\D U_0(0) \\ u_\rho'(0)=(\D u_\rho)'(0)=0.\end{array}\right.\end{align}    Then by ODE local existence theorem $u_\rho$ exists in a neighborhood of the origin. Moreover,  for every $\rho>0$ we have  $u_\rho>U_0$ on $(0,\infty)$, thanks to Lemma \ref{comp}. In fact, $u_\rho(r)\geq \rho+U_0(r)$ on $(0,\infty)$, which implies that    $$\lim_{\rho\to\infty}\int_{\R^3}\frac{dx}{u_\rho^6(x)}=0.$$ Since $0<u_\rho^{-6}\leq U_0^{-6}$ for $\rho\geq0$, by dominated convergence theorem, we have that    the map    $$[0,\infty)\ni\rho\mapsto \int_{\R^3}\frac{dx}{u_\rho^6(x)}$$   is continuous. Hence,   for every $\Lambda\in (0,\Lambda^*]$ there exists a solution $u$ to \eqref{eq-4th} with $\Lambda=\Lambda_u$. 

To prove the converse  we essentially follow \cite{Choi-Xu,Guerra,MR}.  Let   $u$  be a solution of \eqref{eq-4th} for some $\Lambda_u>0$.  Then we have $\D u>0$ in $\R^3$ (see e.g.  \cite[Lemma 2.2]{Choi-Xu}). We set $\bar u(x):=\lambda^\frac{-1}{2}u(\lambda x)$ where $\lambda>0$ is such that $\D \bar u(0)=\D U_0(0)$. Then we have $\Lambda_u=\Lambda_{\bar u}$, and $\bar u=u_\rho$ for some $\rho\in\R$, where $u_\rho$ is the solution to \eqref{urho}. We claim that $\rho\geq 0$. In order to prove the claim we assume by contradiction that $\rho<0$. Then, it follows from Lemma \ref{comp} that   $$\D u_\rho(r)\leq \D U_0(r)-\ve,\quad r\geq 1,$$ for some $\ve>0$.  Therefore, as $\D U_0(\infty)=0$, we have $\D u_\rho(r)\leq-\frac\ve2$ on $(R,\infty)$ for some $R>>1$. In particular, from the  identity \begin{align}\label{formula-1}w(r)=w(0)+\frac{1}{4\pi}\int_0^r\frac{1}{t^2}\int_{B_t}\D w(x)dxdt\quad  \text{for }w\in C^2_{rad}, \end{align}  for some $C_\ve>0$ we obtain     $$u_\rho(r)\leq C-C_\ve r^2\quad\text{ on }(0,\infty), $$   a contradiction as $u_\rho>0$ on $\R^3$.
Thus $\rho\geq 0$, and hence by Lemma \ref{comp} we have $\bar u\geq U_0$ on $(0,\infty)$. This in turn implies that  $\Lambda_{\bar u}\leq\Lambda^*$, and $\Lambda_{\bar u}=\Lambda^*$ if and only if $\bar u=U_0$.
\hfill $\square$
\bigskip


We now move to the proof of Theorem \ref{thm-large}. We start with the following lemma.

\begin{lemma}For $k$ large and  $\ve\in(0,1)$  there exists a positive entire radial solution to     \begin{align}\left\{ \begin{array}{ll}\label{ode-1} \D^3u=-\frac{1}{u^3} \\  u(0)=k\\  \D u(0)=-\ve \\\D^2 u(0)=1\\ u'(0)=(\D u)'(0)=(\D ^2u)'(0)=0.\end{array}\right.\end{align}       Moreover, if $u$ is a  positive entire radial solution to \eqref{ode-1} for some $\ve\in\R$ then  necessarily $\ve\leq \sqrt{\frac{6k}{5}}$, and the solution $u$ satisfies     \begin{align} k-\frac\ve6r^2\leq u(r)\leq k-\frac\ve6r^2+\frac{r^4}{120}\quad \text{on }(0,\infty).\label{upper}\end{align}   \end{lemma}
\begin{proof} It follows from the ODE local existence theorem that for  every $\ve>0$ there exists a  unique positive solution  to \eqref{ode-1} in a neighborhood of the origin.  We let $(0,\delta)$ to be the maximum interval of existence. 

From the   identity \eqref{formula-1} we see  that $\D^2 u$ is strictly monotone decreasing on $(0,\delta)$. Let $\bar \delta \in (0,\delta]$ be the largest number such that     \begin{align}\label{1/2}\D^2u\geq\frac12\quad\text{on }(0,\bar\delta).\end{align}   
Using \eqref{1/2} in  \eqref{formula-1} with $w=\D u$ one obtains     \begin{align*} \D u(r)&\geq -\ve+\frac {1}{12} r^2\quad \text{for } r\in (0,\bar\delta) .\end{align*}     Again by \eqref{formula-1} with $w=u$ we obtain for $r\in (0,\bar\delta)$  \begin{align}\label{u-lower}u(r)\geq k-\frac\ve6r^2+\frac{r^4}{240}\geq\frac k2 +\frac{r^4}{250},\end{align} for $k\geq k_0$ for some $k_0$ sufficiently large and  for every $\ve\in (0,1)$. We can also choose $k_0$ large enough so that     $$\frac{1}{4\pi}\int_0^\infty\frac{1}{t^2}\int_{B_t}\frac{dx}{\left(\frac {k_0}{2}+\frac{|x|^4}{250}\right)^3}dt\leq\frac13.$$     Now we use \eqref{u-lower} in \eqref{formula-1} with $w=\D^2u$ to obtain a lower bound of $\D^2u$. Indeed,  for  $k\geq k_0$ and   $r\in(0,\bar\delta)$    we have     \begin{align} \D^2u(r) &\geq 1-\frac{1}{4\pi}\int_0^r\frac{1}{t^2}\int_{B_t}\frac{dx}{\left(\frac k2+\frac{|x|^4}{250}\right)^3}dt \notag \\  &\geq 1-\frac{1}{4\pi}\int_0^\infty\frac{1}{t^2}\int_{B_t}\frac{dx}{\left(\frac {k_0}{2}+\frac{|x|^4}{250}\right)^3}dt \notag\\ &\geq \frac23.\end{align}      Thus, from the definition of $\bar\delta$ we get that $\bar\delta=\delta$. In particular, \eqref{u-lower} holds on $(0,\delta)$.  This shows that $\delta=\infty$, and we conclude the   first part of the lemma.

To prove \eqref{upper} we let $u$ be a positive entire radial solution to \eqref{ode-1} for some $\ve\in\R$. It follows from \eqref{formula-1} that   $\D^2u$ is strictly monotone decreasing on $(0,\infty) $. Therefore,   as $\D^2u>0$ in $\R^3$ (see e.g. \cite[Lemma 2.2]{Ngo-6}), we get   $$0\leq \D^2 u(\infty)\leq \D^2 u\leq 1\quad\text{ on }(0,\infty).$$     This implies that $\D u$ is monotone increasing on $(0,\infty)$, and  a repeated use of \eqref{formula-1} gives \eqref{upper}. Finally, the upper bound of $u$ in \eqref{upper} and the positivity of $u$ implies that $\ve\leq \sqrt{\frac{6k}{5}}$.

We conclude the lemma.
\end{proof}

\medskip
As a consequence of the above lemma  the number $\ve_k^*$ given by (for $k$ large)     $$\ve_k^*:=\sup\left\{\ve>0:\eqref{ode-1} \text{ has a positive entire solution}\right\} $$    exists, and it satisfies the estimate $\ve_k^*\leq  \sqrt{\frac{6k}{5}}$. Moreover, for every $\ve\in (-\infty,\ve_k^*)$ there exists a positive entire solution to \eqref{ode-1}, thanks to Lemma \ref{comp}.   

\begin{lemma}\label{existence} For $k$ large \eqref{ode-1} has a positive entire solution with $\ve=\ve_k^*$.\end{lemma}
\begin{proof} For simplicity we ignore the subscript $k$ and we write $\ve^*$ instead of $\ve_k^*$.  Let $u$ be the solution to \eqref{ode-1} with $\ve=\ve^*$, and let $(0,R)$ be the maximum interval of existence. We assume by contradiction that $R<\infty$.  Then necessarily  we have    $$\lim_{r\to R-}u(r)=0. $$   
It follows from the definition of $\ve^*$   that there exists a sequence of positive entire solutions  $(u_n)$  to \eqref{ode-1} with $\D u_n(0)\downarrow -\ve^*$.  Then, from the continuous dependence of the solutions  on the  initial data, we have that $u_n\to u$ locally uniformly in $[0,R)$.  In particular, there exists $x_n\to R$ such that $u_n(x_n)\to0$.   We claim that there exists $C>0$ such that    \begin{align}\label{clm1}u_n(r)\leq u_n(x_n)+C(r-x_n)\quad\text{for }x_n\leq r\leq x_n+1.\end{align}    Indeed, as  $0<\D^2 u_n\leq 1$ on $(0,\infty)$, by \eqref{formula-1} we obtain    $$- \ve^*\leq \D u_n(r)\leq r^2\quad\text{on }(0,\infty).$$    This gives $|u_n'|\leq C$ on $(0,R+3)$ for some $C>0$, and  hence we have \eqref{clm1}. Therefore, by \eqref{formula-1}  and together with \eqref{clm1} we get     \begin{align*} \D^2u_n(R+3)&\leq 1-\frac{1}{4\pi}\int_{R+2}^{R+3}\frac{1}{t^2}\int_{x_n<|x|<x_n+1}\frac{dx}{(u_n(x_n)+C(|x|-x_n))^3}dt\\ &\leq 1-\frac{1}{4\pi}\frac{1}{(R+3)^2}\int_{x_n<|x|<x_n+1}\frac{dx}{(u_n(x_n)+C(|x|-x_n))^3} \\ &\xrightarrow{n\to\infty}-\infty,\end{align*}    a contradiction as $\D^2 u_n>0$ on $(0,\infty)$.

We conclude the lemma.
\end{proof}

\begin{lemma}\label{infty}Let $u$ be a positive entire radial solution to \eqref{eq-6th}.  Assume that $\D^2 u(\infty)>0$. Then there exists a positive entire radial solution $v$ to \eqref{eq-6th} such that    $$v(0)=u(0),\quad \D v(0)<\D u(0)\quad\text{and }\D^2v(0)=\D^2u(0).$$     \end{lemma}
\begin{proof} For $\rho>0$ small we consider the initial value problem     \begin{align}\left\{ \begin{array}{ll}\label{ode-v} \D^3v=-\frac{1}{v^3} \\  v(0)=u(0)\\  \D v(0)=\D u(0)-\rho \\ \D^2 v(0)=\D^2 u(0)\\ v'(0)=(\D v)'(0)=(\D^2 v)'(0)=0.\end{array}\right.\end{align}    Since $\D^2 u(\infty)>0$, it follows that $u(r)\geq \delta r^4$ at infinity for some $\delta>0$. Therefore, we can choose $\rho_0>0$ small so that    $$\rho_0 r^2\leq \frac16 u(r)\quad\text{on }(0,\infty).$$      We fix $R_1>>1$ such that     $$\int_{R_1}^\infty\frac{1}{t^2}\int_{B_t}\frac{dx}{u^3(x)}dt<\ve,$$    where $\ve>0$ will be chosen later. By continuous dependence on the initial data we can choose $\rho\in (0,\rho_0)$ sufficiently small such that  the solution $v=v(\rho,u)$ to \eqref{ode-v} exists on $(0,R_1)$ and  it satisfies    $$u-v\leq \ve\quad\text{on }(0,R_1) . $$     We claim that for such $\rho>0$ the solution $v$ exists entirely.

In order to prove the claim we let $R_2>0$ (possibly the largest one) be such that $v\geq \frac u2$ on $(0,R_2)$. (Note that $v\leq u$ on the  common interval of existence, and for $\ve>0$ small enough we have $R_2>R_1$). Then for $0<r<R_2$ we have    \begin{align*} \D^2 v(r)-\D^2u(r)&=O(1)\int_0^r\frac{1}{t^2}\int_{B_t}\frac{u(x)-v(x)}{u(x)v^3(x)}dxdt\\ &\geq -C_1\ve-O(1)\int_{R_1}^{R_2}\frac{1}{t^2}\int_{B_t}\frac{dx}{u^3(x)}dt \\ &\geq -C_2 \ve.\end{align*}     The above estimate and a repeated use of \eqref{formula-1} leads to     \begin{align*}v(r)\geq u(r)-\frac{\rho}{6}r^2 - C_3\ve r^4.\end{align*}     Now we fix $\ve>0$ sufficiently small so that $C_3\ve r^4<\frac 16u(r)$ on $(0,\infty)$. Then we have    $$v(r)\geq\frac23 u(r)\quad\text{on }(0,R_2).$$    Thus, $v\geq\frac u2$ on $(0,R_2)$ implies that $v\geq \frac23u$ on $(0,R_2)$, and hence $R_2=\infty$. 

 This finishes the proof of the claim.  \end{proof}

\medskip

\noindent{\bf\emph{Proof of Theorem \ref{thm-large}}}  Let $(u_k)$ be a sequence of positive entire radial solutions to \eqref{ode-1} with $\ve=\ve_k^*$  as given  by Lemma \ref{existence}. We claim that      \begin{align}\label{to-infty}\int_{\R^3}\frac{dx}{u_k^2(x)}\xrightarrow{k\to\infty} \infty.\end{align}   
First we note  that   $\D^2 u_k(\infty)=0$, that is    \begin{align}\label{10}1=\frac{1}{4\pi}\int_0^\infty\frac{1}{t^2}\int_{B_t}\frac{dx}{u_k^3(x)}dt,\end{align}     which is a consequence of  Lemma \ref{infty}  and the definition of $\ve_k^*$. Moreover,    $$u_k\to \infty\quad\text{locally uniformly in } [0,\infty),$$    thanks to \eqref{upper} and the estimate $\ve_k^*\leq  \sqrt{\frac{6k}{5}}$. Now we consider the following two cases, and we show that \eqref{to-infty} holds in each case.

\medskip

\noindent\textbf{Case 1} $\min_{(0,\infty)}u_k\to\infty$.

Since $u_k\to\infty$ locally uniformly in $\R^3$, from \eqref{10} we obtain     \begin{align*} 1=o(1)+\frac{1}{4\pi}\int_1^\infty\frac{1}{t^2}\int_{B_t}\frac{dx}{u_k^3(x)}dt\leq o(1)+\frac{1}{4\pi\min_{\R^3}u_k}\int_{\R^3}\frac{dx}{u_k^2(x)},\end{align*}    which gives \eqref{to-infty}.

 \medskip

\noindent\textbf{Case 2} $\min_{(0,\infty)}u_k=:u_k(x_k)\leq C$.

 Since $u_k\to\infty$ locally uniformly in $\R^3$, we have $x_k\to\infty$. We claim that    $$u_k(x_k+r)\leq u_k(x_k)+1\quad\text{for }0\leq r\leq \frac{1}{x_k}.$$    In order to prove the claim we  note that $u_k'\geq0$  on $[x_k,\infty)$ and $u_k'(x_k)=0$. Moreover, as $\D^2 u_k\leq \D^2 u_k(0)=1$, by \eqref{formula-1} we have    $$u_k''(x_k+r)+\frac{2}{x_k+r}u_k'(x_k+r)=\D u_k (x_k+r)\leq \frac{1}{6}(x_k+r)^2.$$     Hence, $u_k''(x_k+r)\leq \frac16(x_k+r)^2$,  and by a Taylor expansion, we have our claim. Therefore, as  $x_k\to\infty$, we get     \begin{align*}\int_{x_k<|x|<x_k+\frac{1}{x_k}}\frac{dx}{u_k^2(x)}&\geq \frac{1}{(1+u_k(x_k))^2}\left((x_k+\frac{1}{x_k})^3-x_k^3\right)\\&\geq \frac{3x_k}{(1+u_k(x_k))^2}\\&\xrightarrow{k\to\infty}\infty.\end{align*}    This proves \eqref{to-infty}.

Theorem \ref{thm-large} follows immediately as the integral in \eqref{to-infty} depends continuously on the  initial data, and for every fixed $k$ (large)    $$\int_{\R^3}\frac{dx}{u_{\rho,k}^2(x)}\xrightarrow{\rho\to\infty}0,$$    where $u_{\rho,k}$ is the entire positive solution to \eqref{ode-1} with $\D u_{\rho,k}(0)=\rho>-\ve_k^*$ .
\hfill  $\square$


\begin{thebibliography}{99}


\bibitem{Branson}
\newblock T. P. Branson, 
\newblock Group representations arising from Lorentz conformal geometry, 
\newblock \emph{J. Funct. Anal.}, \textbf{74} (1987), 199-291.


\bibitem{CGS} 
\newblock  L. Caffarelli, B. Gidas and J. Spruck, 
\newblock   Asymptotic symmetry and local behavior of semilinear elliptic equations with critical Sobolev growth,
\newblock \emph{Comm. Pure Appl. Math.},  \textbf{42} (1989), 271-297.

\bibitem{CC}
\newblock S-Y. A. Chang and W. Chen, 
\newblock   A note on a class of higher order conformally covariant equations,
\newblock \emph{Discrete Contin. Dynam. Systems}, \textbf{63} (2001), 275-281.



\bibitem{Choi-Xu}
\newblock Y. S. Choi and X. Xu, 
\newblock  Nonlinear biharmonic equations with negative exponents, 
\newblock \emph{J. Diff. Equations}, \textbf{246} (2009), 216-234.

\bibitem{Ngo2}
\newblock  T. V. Duoc and Q. A.  Ng\^o,  
\newblock  A note on positive radial solutions of $\Delta^2 u+u^{-q}=0$ in $\R^3$ with exactly quadratic growth at infinity,
\newblock \emph{Diff. Int. Equations},  \textbf{30} (2017), no. 11-12, 917-928.

\bibitem{Ngo-6} 
\newblock  T. V. Duoc and Q. A. Ng\^o, 
\newblock  Exact growth at infinity for radial solutions of $\Delta^3u+u^{-q}=0$ in $\mathbb{\R}^3$, 
\newblock  Preprint (2017), \url{ftp://file.viasm.org/Web/TienAnPham-17/Preprint_1702.pdf}.

\bibitem{Farina} 
\newblock  A. Farina   and  A. Ferrero, 
\newblock Existence and stability properties of entire solutions to the polyharmonic equation $(-\Delta u)^m=e^u$  for any $m>1$, 
\newblock\emph{Ann. Inst. H. Poincar\'e Anal. Non Lin\'eaire},  \textbf{33} (2016), 495-528.

\bibitem{F-X}
\newblock  X. Feng and  X.  Xu, 
\newblock   Entire solutions of an integral equation in $\R^5$, 
\newblock \emph{ISRN Math. Anal.},  (2013), Art. ID 384394, 17 pp.  

\bibitem{Guerra}
\newblock  I. Guerra, 
\newblock  A note on nonlinear biharmonic equations with negative exponents,
\newblock \emph{J. Differential Equations}, \textbf{253} (2012), 3147-3157.

 \bibitem{HY}
 \newblock  X. Hunag and D. Ye, 
 \newblock   Conformal metrics in $\mathbb{R}^{2m}$ with constant $Q$-curvature and arbitrary volume,
 \newblock \emph{Calc. Var. Partial Differential Equations},   \textbf{54} (2015), 3373-3384.

 \bibitem{H}
 \newblock  A. Hyder, 
 \newblock Conformally Euclidean metrics on $\R^n$ with arbitrary total Q-curvature,
 \newblock \emph{Anal. PDE}, \textbf{10} (2017), no. 3, 635–-652.


 \bibitem{HW}
 \newblock A. Hyder and  J. Wei,
 \newblock  Non-radial solutions to a biharmonic equation with negative exponent,
 \newblock  Preprint (2018), \url{http://www.math.ubc.ca/~ali.hyder/W/HW.pdf}.

\bibitem{Lai} 
\newblock B. Lai, 
\newblock  A new proof of I. Guerra's results concerning nonlinear biharmonic equations with negative exponents,
\newblock \emph{J. Math. Anal. Appl.}  \textbf{418} (2014), 469-475.

\bibitem{Li} 
\newblock  Y. Li,
\newblock  Remarks on some conformally invariant integral equations: The method of moving spheres,
\newblock \emph{J. Eur. Math. Soc.}, \textbf{6} (2004), 1-28.

\bibitem{Lin} 
\newblock  C. S. Lin, 
\newblock   A classification of solutions of a conformally invariant fourth order equation in $\mathbb{R}^n$,
\newblock \emph{Comment. Math. Helv.},  \textbf{73} (1998), 206-231.


\bibitem{M}
\newblock  L. Martinazzi, 
\newblock  Conformal metrics on $\R^{2m}$ with constant Q-curvature and large volume,
\newblock \emph{Ann. Inst. H. Poincar\'e Anal. Non Lin\'eaire},  \textbf{30} (2013), no. 6, 969-–982.

\bibitem{MR}
\newblock  P. J. McKenna and W. Reichel, 
\newblock  Radial solutions of singular nonlinear biharmonic equations and applications to conformal geometry,
\newblock \emph{Electron. J. Differential Equations}, \textbf{37} (2003), 1-13.

\bibitem{Ngo}
\newblock  Q. A. Ng\^o,
\newblock  Classification of entire solutions of $(-\Delta)^n u+u^{4n-1}=0$ with exact linear growth at infinity in $\mathbb{R}^{2n-1}$,
\newblock \emph{Proc. Amer. Math. Soc.}, \textbf{146} (2018), no. 6, 2585-2600.


\bibitem{WY}
\newblock J. Wei and D. Ye, 
\newblock  Nonradial solutions for a conformally invariant fourth order equation in $\R^4$,
\newblock \emph{Calc. Var. Partial Differential Equations}, \textbf{32} (2008), no. 3, 373-–386.

\bibitem{Xu} 
\newblock  X. Xu, 
\newblock  Exact solutions of nonlinear conformally invariant integral equations in $\mathbb{R}^3$,
\newblock \emph{Adv.  Math.}, \textbf{194} (2005), 485-503.

\bibitem{XY} 
 \newblock  X. Xu and P. Yang, 
 \newblock  On a fourth order equation in $3$-$D$, 
 \newblock \emph{ESAIM:  Control Optim. Calc. Var.},  \textbf{8} (2002), 1029-1042.

\bibitem{Yang-Zhu}
\newblock  P. Yang and M.  Zhu, 
\newblock  On the Paneitz energy on standard three sphere,
\newblock \emph{ESAIM: Control Optim. Calc. Var.}, \textbf{10}  (2004), no. 2, 211-223.

\end{thebibliography}
\end{document}